\title{\vspace{-0.7cm}Freiman homomorphisms on sparse random sets}
\author{D. Conlon\thanks{Mathematical Institute, Woodstock Road, Oxford OX2 6GG, UK.
E-mail: {\tt david.conlon@maths.ox.ac.uk}. Research supported by a
Royal Society University Research Fellowship and by ERC Starting Grant 676632.} \and W. T. Gowers\thanks{Department of Pure Mathematics and Mathematical Statistics,
Wilberforce Road, Cambridge CB3 0WB, UK. Email: {\tt
w.t.gowers@dpmms.cam.ac.uk}. Research supported by a Royal Society 2010 Anniversary Research Professorship.}}

\documentclass[11pt]{article}
\usepackage{amsfonts,amssymb,amsmath,latexsym,color}

\newenvironment{proof}
      {\medskip\noindent{\bf Proof.}\hspace{1mm}}
      {\hfill$\Box$\medskip}

\def\qed{\ifvmode\mbox{ }\else\unskip\fi\hskip 1em plus 10fill$\Box$}

\def\eps{\epsilon}
\def\E{\mathbb{E}}
\def\Z{\mathbb{Z}}
\def\R{\mathbb{R}}
\def\T{\mathbb{T}}
\def\C{\mathbb{C}}

\def\P{\mathbb{P}}

\def\b{\beta}

\def\d{\delta}
\def\D{\Delta}
\def\e{\epsilon}

\def\ra{\rightarrow}
\def\ol{\overline}
\def\hf{\hat{f}}
\def\hg{\hat{g}}
\def\var{\mathrm{var}}
\def\cf{\mathcal F}
\def\supp{\mathop{\mbox{supp}}}

\def\sp#1{\langle #1\rangle}

\newtheorem{theorem}{Theorem}[section]

\newtheorem{lemma}[theorem]{Lemma}

\newtheorem{corollary}[theorem]{Corollary}

\newtheorem{definition}[theorem]{Definition}

\begin{document}
\date{}
\maketitle

\begin{abstract}
A result of Fiz Pontiveros shows that if $A$ is a random subset of $\Z_N$ where each element is chosen independently with probability $N^{-1/2+o(1)}$, then with high probability every Freiman homomorphism defined on $A$ can be extended to a Freiman homomorphism on the whole of $\Z_N$. In this paper we improve the bound to $CN^{-2/3}(\log N)^{1/3}$, which is best possible up to the constant factor.
\end{abstract}

\section{Introduction}

A \textit{Freiman homomorphism} from a subset $A\subset\Z_N$ to $\Z_N$ is a function $\phi:A\ra\Z_N$ with the property that $\phi(a)+\phi(b)=\phi(c)+\phi(d)$ whenever $a,b,c$ and $d$ are elements of $A$ with $a+b=c+d$. If $\phi$ is an affine function, meaning that there exist $r$ and $s$ such that $\phi(a)=ra+s$ for every $a \in A$, then clearly $\phi$ is a Freiman homomorphism. It is also easy to prove that every Freiman homomorphism from $\Z_N$ to $\Z_N$ is affine. From this it follows that a Freiman homomorphism defined on $A$ is affine if and only if it can be extended to a Freiman homomorphism on the whole of $\Z_N$.

Over the last ten years or so, a number of results have been proved that show that important combinatorial and additive properties of structures such as complete graphs and finite Abelian groups are preserved when one passes to random subsets of those structures of surprisingly low density. We will not attempt an exhaustive summary of this area here, referring the reader instead to the papers~\cite{BMS15, CG10, CGSS14, FRS10, ST15, S10} and the surveys~\cite{C14, RS13}. In the light of these developments, it is natural to ask whether with high probability every Freiman homomorphism defined on a random subset of $\Z_N$ must be affine. It turns out to be an easy exercise to show that if the elements of $A$ are chosen independently with probability $CN^{-1/3}$, then this is indeed the case. Recently, Fiz Pontiveros~\cite{F13} proved the significantly harder result that this remains true if the elements of $A$ are chosen with probability $N^{-1/2+o(1)}$. 

In the other direction, it is easy to see that the result is false if the elements of $A$ are chosen with probability $N^{-2/3}/2$. Indeed, for each element $a$, the expected number of triples $(b,c,d)\in A^3$ with $a+b=c+d$ and no two of $a,b,c$ and $d$ equal is at most $N^2(N^{-2/3}/2)^3=1/8$, from which it follows that there is a high probability that there will be an ``isolated" element $a\in A$ that belongs to no non-trivial quadruple $a+b=c+d$. But that implies that however we choose $\phi(a)$ we will not violate the condition for $\phi$ to be a Freiman homomorphism.

With a little more effort, one can show that even if we choose elements independently with probability $c (\log N)^{1/3}N^{-2/3}$, for a sufficiently small positive constant $c$, there is still a significant probability that some of the elements of $A$ will be isolated in this sense. Therefore, the best result that one can hope for is a probability on the order of $(\log N)^{1/3}N^{-2/3}$. We shall obtain a bound of this form.

\begin{theorem} \label{thm:main1}
There exists a positive constant $C$ such that if $A$ is a random subset of $\Z_N$ where each element is chosen independently with probability $C (\log N)^{1/3}N^{-2/3}$, then, with high probability, every Freiman homomorphism from $A$ to $\Z_N$ is affine.
\end{theorem}

It is also possible to show by similar methods that if the elements of $A$ are chosen with probability $CN^{-2/3}$, then, with high probability, for every Freiman homomorphism $\phi$ defined on $A$ there is an affine function that agrees with $\phi$ on most of $A$. We have not shown this in detail, because the additional factor of $(\log N)^{1/3}$ that is needed for our main result to be true also simplifies other aspects of the proof. 

Theorem~\ref{thm:main1} is somewhat unexpected because the bound obtained by Fiz Pontiveros is a natural boundary for his method and could lead one to think that his result is best possible. His proof is based on the following idea. If $\phi$ is a Freiman homomorphism defined on $A$, then one can define a function $\psi$ on $A-A$ by taking $\psi(a-b)$ to be $\phi(a)-\phi(b)$. This is well-defined because if $a-b=c-d$, then $\phi(a)-\phi(b)=\phi(c)-\phi(d)$. Fiz Pontiveros proves that $\phi$ is an affine function by noting that $A-A$ is, with high probability, the whole of $\Z_N$ and then proving that $\psi$ is an additive function: that is, $\psi(x+y)=\psi(x)+\psi(y)$ for every $x$ and $y$. This implies that $\psi$ is linear and hence that $\phi$ is affine. 

However, the difference set $A-A$ trivially ceases to be the whole of $\Z_N$ once $p$ goes below $N^{-1/2}$, which is why $N^{-1/2}$ is as far as Fiz Pontiveros's argument will go. Indeed, one can say more. For any non-zero $x$, the expected number of pairs $(a,b)\in A^2$ such that $a-b=x$ is $p^2N$, so if $p$ is substantially less than $N^{-1/2}$, as it will be for us, then this expectation is substantially less than 1. Moreover, even when $x$ \textit{is} an element of $A-A$, this will almost always happen in just one way, so it seems as though it ought to be very hard for the values of $\phi$ on one part of $A$ to influence the values it takes on other parts. However, all is not quite lost, because if $p=C(\log N)^{1/3}N^{-2/3}$, then for each $a\in\Z_N$, and in particular for each $a\in A$, the expected number of triples $(b,c,d)\in A^3$ such that $a+b=c+d$ and no two of $a$, $b$, $c$ and $d$ are equal is roughly $C^3 \log N$. Therefore, a typical element of $A$ will feel the influence from the rest of the set. Nevertheless, it is quite surprising that one can obtain a global result when the average number of such triples is so small.

The proof of Theorem~\ref{thm:main1} has three parts. To begin, we prove a transference principle similar to that used in~\cite{CG10} to prove a number of analogues of combinatorial theorems in sparse random sets. One of the main tools we used there was the finite-dimensional Hahn--Banach separation theorem. In this paper, we shall also use the Hahn--Banach theorem, but this time it is the complex version of the theorem that will be useful to us, and the way we use it will be different. We then use this complex-valued transference principle to prove that any Freiman homomorphism from $A$ to $\Z_N$ agrees on most of $A$ with an affine function. In~\cite{CG10}, our results were usually straightforward corollaries of the transference principle. Here, a number of additional ideas are needed to make the argument work. Finally, we conclude with a short argument showing that any function which is affine on most of $A$ is actually affine on all of $A$. Since our proof does not depend in a serious way on the structure of $\Z_N$, we shall prove the following more general result.

\begin{theorem} \label{thm:main2}
There exists a positive constant $C$ such that if $G$ is an Abelian group of order $n$ and $U$ is a random subset of $G$ where each element is chosen independently with probability $C(\log n)^{1/3}n^{-2/3}$, then, with high probability, every Freiman homomorphism from $U$ to an Abelian group $H$ can be extended to a Freiman homomorphism defined on all of $G$. 
\end{theorem}

Note that a Freiman homomorphism defined on all of $G$ is simply an affine homomorphism: that is, a map of the form $g\mapsto \phi(g)+h$ for some group homomorphism $\phi$ and some $h\in H$. Thus, the result is saying that all Freiman homomorphisms defined on $U$ are restrictions of affine homomorphisms.

\section{Consequences of the complex finite-dimensional Hahn--Banach theorem}

The version of the Hahn--Banach theorem that we shall rely on is the following. Recall that, given a norm $\|.\|$ on a vector space $X$, the dual norm $\|.\|^*$ of $\|.\|$ is a norm on the collection of linear functionals $\phi$ acting on $X$, given by
\[\|\phi\|^* = \sup\{|\phi(x)|: \|x\| \leq 1\}.\]

\begin{theorem} \label{hahnbanach}
Let $\|.\|$ be a norm on $\C^n$ for some positive integer $n$ and let $\phi$ be a linear functional defined on a subspace $X$ of $\C^n$. Then $\phi$ can be extended to a linear functional $\psi$ on $\C^n$ with $\|\psi\|^*=\|\phi\|^*$. 
\end{theorem}

In particular, this theorem has the following corollary. Recall that a subset of $\C^n$ is \textit{absolutely convex} if it is convex and closed under multiplication by scalars of modulus 1. 

\begin{corollary} \label{hbcor}
Let $K$ and $L$ be two closed bounded absolutely convex subsets of $\C^n$ and suppose that $0$ belongs to the interior of $K+L$. Let $v\in\C^n$ be a vector that does not belong to $K+L$. Then there exists a linear functional $\psi$ such that $\psi(v)>1$ and $|\psi(w)|\leq 1$ for every $w\in K\cup L$.
\end{corollary}

\begin{proof}
Define a norm $\|.\|$ on $\C^n$ by $\|v\|=\inf\{|\lambda|+|\mu|:v\in\lambda K+\mu L\}$. The absolute convexity of $K$ and $L$ guarantees that this is a norm. Then $\|v\|>1$, since otherwise for every $\e>0$ we would be able to find $x\in K$ and $y\in L$ and complex scalars $\lambda$ and $\mu$ with $|\lambda|+|\mu|\leq 1+\e$ such that $v=\lambda x+\mu y$. By the compactness of $K$ and $L$ and the fact that both are closed under multiplication by scalars of modulus at most $1$, it would then follow that $v\in K+L$. 

Since $\|v\|>1$, the linear functional $\phi$ defined on the 1-dimensional subspace generated by $v$ given by $\phi(\lambda v)=\lambda\|v\|$ has dual norm $1$, and $\phi(v)=\|v\|>1$. By Theorem \ref{hahnbanach}, we can extend $\phi$ to a linear functional $\psi$ defined on all of $\C^n$ such that $\|\psi\|^*\leq 1$. From this last property, we see that if $w\in K$ then $|\psi(w)|\leq\|w\|\leq 1$, and similarly if $w\in L$. This proves the lemma.
\end{proof}

The main result of this section is the following further corollary. We will use the fact that every linear functional on a Hilbert space is of the form $\sp{\, \cdot \, ,\phi}$ for some $\phi$, where the inner product of two functions $f, g: X \rightarrow \C$ is given by $\sp{f, g} = \E_x f(x) \overline{g(x)}$. By saying that a function is a \emph{measure} on a finite set $X$, we mean that it is a non-negative function from $X$ to $\R$.

\begin{corollary} \label{maintool}
Let $\mu$ and $\nu$ be measures on a finite set $X$ and let $\|.\|$ be a norm on $\C^X$. Suppose that $|\sp{\mu-\nu,|\phi|}|\leq\e$ for every function $\phi$ such that $\|\phi\|^*\leq\eta^{-1}$. Let $f$ be a function such that $|f|\leq\mu$. Then there exists a function $g$ such that $0\leq|g|\leq(1-\e)^{-1}\nu$ and $\|f-g\|\leq\eta$.
\end{corollary}

\begin{proof}
Suppose that we cannot find such a $g$. Then $f\notin K+L$, where $K=\{g:0\leq|g|\leq(1-\e)^{-1}\nu\}$ and $L=\{h:\|h\|\leq\eta\}$. Since both $K$ and $L$ are closed, bounded, and absolutely convex, it follows from Corollary \ref{hbcor} that there exists $\phi$ such that $\sp{f,\phi}>1$, $|\sp{g,\phi}|\leq(1-\e)$ whenever $0\leq|g|\leq\nu$, and $|\sp{h,\phi}|\leq 1$ whenever $\|h\|\leq\eta$. The third condition tells us that $\|\phi\|^*\leq\eta^{-1}$. The second tells us that $\sp{\nu,|\phi|}\leq 1-\e$, since the function $g$ that maximizes $|\sp{g,\phi}|$ subject to the constraint that $|g|\leq\nu$ is the function $g(x)=\nu(x)e^{i \mathrm{arg}(\phi(x))}$, and for that $g$ we have $\sp{g,\phi}=\sp{\nu,|\phi|}$. From these facts and our hypothesis we deduce that
\begin{equation*}
1<\sp{f,\phi}\leq\sp{|f|,|\phi|}\leq\sp{\mu,|\phi|}\leq\e+\sp{\nu,|\phi|}\leq 1,
\end{equation*}
a contradiction.
\end{proof}

\section{A norm to which the Hahn--Banach argument will be applied}

Let $G$ be an Abelian group of order $n$ and let $\Gamma$ be the set of all non-degenerate additive quadruples in $G$. That is, $\Gamma$ is the set of all quadruples $(x,y,z,w)$ such that $x+y=z+w$ and neither $x$ nor $y$ is equal to $z$ or $w$. For every function $f$ we define a quantity $M(f)$ to be $\E_{(x,y,z,w)\in\Gamma}f(x)f(y)\ol{f(z)f(w)}$.

Before we continue, we make a quick remark. It is important to consider non-degenerate additive quadruples only, since in a random set of density $n^{-2/3}$, the non-degenerate additive quadruples are swamped by the degenerate ones. Indeed, the number of non-degenerate additive quadruples is approximately $n^{3-8/3}=n^{1/3}$, whereas the number of degenerate additive quadruples is around $(n^{1/3})^2=n^{2/3}$. (The density at which there are roughly equal numbers of degenerate and non-degenerate additive quadruples is $n^{1/2}$, which is another reason for the natural-seeming barrier there.) At first this appears to be a serious problem, because the quantity $M(f)$ that we have defined is not the fourth power of a norm. However, it turns out not to matter, because the norm we use is constructed in a different way from the $U^2$ norm.

Let $U$ be a random subset of $G$ with characteristic measure $\mu$. (The \textit{characteristic measure} of a set is its characteristic function divided by its density. Thus, $\mu$ is zero outside $U$, constant inside $U$, and has average value 1.) We shall show that, with high probability, for every function $f$ from $G$ to $\C$ such that $|f|\leq\mu$ there exists a function $g$ with the following three properties:

\begin{enumerate}
\item $\|g\|_\infty\leq 1$;
\item $M(g)\approx M(f)$;
\item $\langle g,\tau\rangle\approx \langle f,\tau\rangle$ for every character $\tau:G\to\mathbb{C}$.
\end{enumerate}

The approach we use to obtain this transference result is closely related to the approach we used in \cite{CG10}, though in that paper we used the real Hahn--Banach theorem. We begin by defining a norm $\|.\|$ with the property that if $\|f-g\|$ is small, then conclusions 2 and 3 above hold. For this, we begin with a simple observation. Write $M(f,g,h,k)$ for the quantity $\E_{(x,y,z,w)\in\Gamma}f(x)g(y)\ol{h(z)k(w)}$. (Thus, $M(f)$ can be thought of as shorthand for $M(f,f,f,f)$.) Then $M$ is additive in all four variables, so
\begin{equation*}
M(f)-M(g)=M(f-g,f,f,f)+M(g,f-g,f,f)+M(g,g,f-g,f)+M(g,g,g,f-g).
\end{equation*}
Each of the four expressions on the right-hand side can be regarded as the inner product of $f-g$ with another function. For example, $M(g,g,f-g,f)=\langle h,f-g\rangle$, where $h(z)=\E_{(x,y,z,w)\in\Gamma}g(x)g(y)\ol{f(w)}$. This motivates the following definitions. First, for any $x\in\Z_N$, let us define $\Gamma_x$ to be the set of triples $(y,z,w)$ such that $(x,y,z,w)\in\Gamma$. These triples satisfy the equation $z+w-y=x$, but they also satisfy the non-degeneracy condition.

Now define a \textit{basic anti-uniform function} to be any function from $G$ to $\C$ of the form 
\[u(x)=\E_{(y,z,w)\in\Gamma_x}\ol{h_1(y)}h_2(z)h_3(w),\]
 where each $h_i$ either satisfies $\|h_i\|_\infty\leq 1$ or $|h_i(v)| \leq \mu(v)$ for every $v$. Then each of the four terms on the right-hand side of the equation above is the inner product of $f-g$ with some basic anti-uniform function (with $f-g$ either on the left or on the right). Therefore, if all such inner products are small, then $M(f)\approx M(g)$. If $h_1=h_2=h_3=\tau$ for some character $\tau$, then $u=\tau$, since for every $(y,z,w)\in\Gamma_x$ we have $\ol{\tau(y)}\tau(z)\tau(w)=\tau(x)$. Therefore, all characters are basic anti-uniform functions and we also obtain the third condition.

We now define a norm $\|.\|$ by setting $\|f\|$ to be the maximum of $|\sp{f,u}|$ over all basic anti-uniform functions $u$. From the discussion above, the following lemma follows easily.

\begin{lemma}
Let $G$ be a finite Abelian group, let $f$ and $g$ be functions from $G$ to $\C$, and let $\|.\|$ be the norm just defined. Then $|M(f)-M(g)|\leq 4\|f-g\|$.
\end{lemma}

\begin{proof}
As commented above, $M(f)-M(g)$ can be written as a sum of four terms, each of which is an inner product of $f-g$ with a basic anti-uniform function.
\end{proof}

We now want to apply Corollary \ref{maintool}. To do that, we need an expression for the dual norm of the norm $\|.\|$ we have just defined. The following lemma is a standard fact, which can be proved easily with the help of the Hahn--Banach theorem.

\begin{lemma} \label{dualnorm}
Let $X$ be a finite set and let $\|.\|$ be a norm on $\C^X$ defined by a formula of the form $\|f\|=\max\{|\sp{f,\psi}|:\psi\in\Psi\}$. Then the dual norm $\|.\|^*$ is given by the formula 
\[\|\phi\|^*=\inf\{\sum_{i=1}^n|a_i|:\phi=\sum_{i=1}^na_i\psi_i,\ \psi_i\in\Psi\}.\]
\end{lemma}

The fact that $\C^X$ is finite-dimensional implies also that this infimum is attained. So the condition in Corollary \ref{maintool} that $\|\phi\|^*\leq\eta^{-1}$ implies that $\phi$ can be written as a linear combination of basic anti-uniform functions with the absolute values of the coefficients summing to at most $\eta^{-1}$.

We want to apply Corollary \ref{maintool} with $\mu$ being the characteristic measure of a sparse random set $U$ and with $\nu$ being the constant function 1. Therefore, we need to establish that $|\sp{\mu-1,|\phi|}|$ is small whenever $\phi$ is of the form just described. Before we start on this, we must find a way to deal with the fact that we are looking at $|\phi|$, which cannot be described as easily as $\phi$. This we do with the help of the following lemma, which is a special case of the Stone--Weierstrass theorem.

\begin{lemma} \label{weierstrass}
For every pair of real numbers $C,\e>0$ there exists a polynomial $P$ in $z$ and $\overline{z}$ that uniformly approximates the function $|z|$ to within $\e$ on the disc $\{z:|z|\leq C\}$.
\end{lemma}

From this we obtain a further reduction of what we hope to prove.

\begin{corollary} \label{products}
For every $\e>0$ and every real number $C$ there exist $\d>0$ and a positive integer $k$ with the following property. Let $\mu$ be a measure on $G$ with $\E\mu=1$ and let $\|.\|$ be the norm defined earlier on $\C^G$. Suppose that $|\sp{\mu-1,\xi}|\leq\d$ for every $\xi$ that can be written as a product of at most $k$ basic anti-uniform functions. Suppose also that $\|u\|_\infty\leq 2$ for every basic anti-uniform function. Then $|\sp{\mu-1,|\phi|}|\leq\e$ for every function $\phi$ such that $\|\phi\|^*\leq C$.
\end{corollary}

\begin{proof} Let $P$ be a polynomial in $z$ and $\overline{z}$ that approximates $|z|$ to within $\e/3$ on the closed disc of radius $2C$. Then $\|P\circ\phi-|\phi|\|_\infty\leq\e/3$ for every function $\phi:G\to\C$ with $\|\phi\|_\infty\leq 2C$. By assumption, $\|u\|_\infty\leq 2$ for every basic anti-uniform function $u$, so, by Lemma \ref{dualnorm}, $\|\phi\|_\infty\leq 2C$ whenever $\|\phi\|^*\leq C$. 

By Lemma \ref{dualnorm} again, and also the remark following it, if $\|\phi\|^*\leq C$ then we can express $\phi$ as a linear combination $\sum_{i=1}^na_i\psi_i$, where the $\psi_i$ are basic anti-uniform functions. Since the pointwise complex conjugate of a basic anti-uniform function is also a basic anti-uniform function, it follows that $P\circ\phi$ is a linear combination of products of at most $k$ basic anti-uniform functions, where $k$ is the degree of $P$ and the sum of the absolute values of the coefficients in the linear combination is bounded above by a constant $M$ that depends on $P$ and $C$ only. (A bound that works is $Q(C)$, where $Q$ is the polynomial obtained from $P$ by replacing each coefficient by its absolute value.)  

If we set $\d=\e/3M$, then we obtain the upper bound $|\sp{\mu-1,P\circ\phi}|\leq\e/3$ for every $\phi$ with $\|\phi\|^*\leq C$. Since $\|P\circ\phi-|\phi|\|_\infty\leq\e/3$ and $\|\mu\|_1=\|1\|_1=1$, it follows that $|\sp{\mu-1,|\phi|}|\leq\e$, as claimed.
\end{proof}

The main task ahead of us, therefore, is to prove that under suitable circumstances we have the hypothesis that $|\langle\mu-1,\xi\rangle|\leq\d$ for every product $\xi$ of at most $k$ basic anti-uniform functions. For this we shall need some probabilistic lemmas, proved in the next section, followed by an argument similar to one in \cite{CG10} that established a result of this type. However, the argument in this paper is substantially simpler, because the factor of $(\log n)^{1/3}$ in our probability allows us to prove that certain functions are uniformly bounded rather than bounded almost everywhere. (We stress that we do not introduce the logarithmic factor merely to simplify the proof: for this problem it gives the correct bound up to a constant factor.) 

\section{Some probabilistic lemmas}

We begin with the standard Chernoff bound. 

\begin{lemma} \label{singleset}
Let $X$ be a set of size $n$, let $0<\d\leq 1$, let $p\in[0,1]$ and let $U$ be a random subset of $X$ where each element is chosen independently with probability $p$. Then 
\[\P\bigl[\bigl||U|-pn\bigr|>\d pn\bigr]\leq 2\exp(-\d^2 pn/3).\] 
\end{lemma}
Amongst other things, this lemma allows us to ignore the difference between the characteristic measure $\mu$ of $U$, defined by $\mu(x) = |X|/|U|$ if $x \in U$ and $0$ otherwise, and its associated measure, defined by $\mu'(x) = p^{-1}$ if $x \in U$ and $0$ otherwise. For example, in the next section, we will actually prove that with high probability $|\langle\mu'-1,\xi\rangle|\leq\d$ for every product $\xi$ of at most $k$ basic anti-uniform functions built relative to $\mu'$, but the result for the characteristic measure is an immediate corollary.

\begin{lemma} \label{doubleconvolution}
Let $U_1$ and $U_2$ be two random subsets of an Abelian group $G$ of size $n$, each with elements chosen independently with probability $p$. Then the probability that there exists an element of $G$ that can be written in $s$ ways as $u_1+u_2$ with $u_1\in U_1$ and $u_2\in U_2$ is at most $n^{s+1}p^{2s}$.
\end{lemma}

\begin{proof}  
Fix $x\in G$ and for each $u$ let $E(u)$ be the event that $u\in U_1$ and $x-u\in U_2$. Then the events $E(u)$ are independent and each holds with probability $p^2$, so the expected number of $s$-tuples $(u_1, \dots,u_s)$ with distinct elements such that $E(u_1), \dots, E(u_s)$ all hold is at most $n^s p^{2s}$. Therefore, the probability that $x$ can be written in $s$ ways as the sum of an element of $U_1$ and an element of $U_2$ is at most $n^s p^{2s}$, so the probability that \emph{some} $x$ can be written in $s$ ways is at most $n^{s+1} p^{2s}$, as claimed.
\end{proof}

For the next lemma, we recall that the convolution of two functions $f, g : G \rightarrow \C$ is given by $f*g(x) = \E_{y + z =x} f(y) g(z)$. 

\begin{lemma} \label{tripleconvolution}
For every $0<\e\leq 1$, there exists a constant $C$ with the following property. Let $G$ be an Abelian group of order $n$ and let $U_1,U_2$ and $U_3$ be independent random subsets of $G$ with each element chosen (for each $U_i$) independently with probability $p=Cn^{-2/3}(\log n)^{1/3}$. For each $i$, let $\mu_i=p^{-1}\chi_{U_i}$, where $\chi_{U_i}$ is the characteristic function of $U_i$, and let $\mu_i^-(x)=\mu_i(-x)$. Then, with probability $1-o(n^{-3})$, every value of the function $\mu_1*\mu_2*\mu_3^-$ lies between $1-\e$ and $1+\e$. 
\end{lemma}

\begin{proof}
Let $x\in G$. Then $\mu_1*\mu_2*\mu_3^-(x)$ is equal to $p^{-3}$ times the number of ways of writing $x$ as $u_1+u_2-u_3$ with $u_i\in U_i$ for each $i$. By Lemma \ref{doubleconvolution}, there is a probability of $1-o(n^{-3})$ that the maximum value of $\chi_{U_1}*\chi_{U_2}$ is at most $12$. For convenience, write $\eta$ for $\e/4$. Then, by Lemma \ref{singleset}, there is a probability of $1-o(n^{-3})$ that both $U_1$ and $U_2$ have sizes between $(1-\eta)pn$ and $(1+\eta)pn$. Let us fix $U_1$ and $U_2$ with these properties and consider our random choice of $U_3$. 

The number of ways of writing $x$ as $u_1+u_2-u_3$ is equal to $\sum_{u\in U_3}\chi_{U_1}*\chi_{U_2}(x+u)$. This we can represent as a sum of independent random variables $X_u$, where $X_u=\chi_{U_1}*\chi_{U_2}(x+u)$ with probability $p$ and 0 with probability $1-p$. The expectation of $\sum_uX_u$ is $n^{-1}|U_1||U_2|\E|U_3|$, which lies between $(1-\eta)^2p^3n^2$ and $(1+\eta)^2p^3n^2$. From this and the fact that each $X_u$ is bounded above by 12, $\sum_u\var(X_u)\leq 12(1+\eta)^2p^3n^2\leq 20p^3n^2$. Therefore, by Bernstein's inequality, 
\[\P\bigl[\bigl|\sum_uX_u-\E\sum_uX_u\bigr|>\eta p^3n^2\bigr]\leq 2\exp(-\eta^2p^6n^4/2(20p^3n^2+ 4\eta p^3n^2))\leq 2\exp(-\eta^2p^3n^2/48).\]
But $p^3n^2=C^3\log n$, so if we set $C=8/\eta^{2/3}$ then this upper bound is at most $\exp(-8\log n)=o(n^{-4})$. 

We are more or less done, but we need to renormalize. Note that 
\[\E_{u_1+u_2-u_3=x}\mu_1(u_1)\mu_2(u_2)\mu_3(u_3)=n^{-2}p^{-3}\sum_uX_u,\]
so we have shown that with probability $1-o(n^{-3})$, every value of $\mu_1*\mu_2*\mu_3^-$ lies between $(1-\eta)^2-\eta$ and $(1+\eta)^2+\eta$. Since $\eta\leq 1$, these values lie between $1-4\eta$ and $1+4\eta$, which proves the result.
\end{proof}

\section{Basic anti-uniform functions are approximately contained in the convex hull of a small set}

We now return to our task of proving that with high probability $\langle\mu-1,\xi\rangle$ is small whenever $\xi$ is a product of not too many basic anti-uniform functions. A difficulty with doing this is that the definition of a basic anti-uniform function depends on the measure $\mu$, but this turns out to be less of a problem than it looks: it will be enough to prove the result when the random measure $\mu$ and all the measures used to define the different basic anti-uniform functions are independent. We shall do this first and then give a standard argument that shows why it is enough.

As for the result when all the measures are independent, the technique here, as in \cite{CG10}, is to prove that there is a set $\cf$ of bounded functions that is not too large such that every product $\xi$ of basic anti-uniform functions can be approximated by a convex combination $\omega$ of functions in $\cf$. In \cite{CG10} we had to give a somewhat complicated definition of ``can be approximated by", but here it is simply a uniform approximation. If $\cf$ is small enough and if for every bounded function $g$ the inner product $\langle\mu-1,g\rangle$ is small with high probability, then a union bound will tell us that with high probability $\langle\mu-1,\phi\rangle$ is small for every $\phi\in\cf$ and hence for every $\phi$ in the convex hull of $\cf$.

The proof is slightly complicated by the fact that our basic anti-uniform functions are convolutions of two kinds of functions: functions that are bounded above by 1 and functions that are bounded above by the characteristic measure of a sparse random set. Let us consider first the functions of the latter kind. For these, our argument is based on a simple observation. Suppose that $f$ is a function defined on a finite Abelian group and $\theta\in [0,1]$. Then we define the $\theta$-\emph{random restriction} $R_\theta f$ of $f$ to be a random function $g$ where for each $x$ we have $g(x)=f(x)/
\theta$ with probability $\theta$ and $g(x)=0$ with probability $1-\theta$, where all these events are independent. It is important that $R_\theta f$ is not a fixed function but a random one. We shall adopt as a notational convention that if we have functions $R_\theta f_1,\dots,R_\theta f_k$ then all the random values $R_\theta f_i(x)$ are independent. That is, the random restrictions of the different $f_i$ are made independently. We adopt this convention even if some of the $f_i$ are equal. For example, $R_\theta f*R_\theta f$ denotes the convolution of a random restriction of $f$ with another, independently chosen, random restriction of $f$.

We shall want to take averages over random functions. To avoid confusion with averages of the form $\E_xf(x)$ we shall write $\E_\Sigma$ for the average over the functions themselves.

Our simple observation is that many quantities defined in terms of random restrictions average to the corresponding quantities for the original functions. The underlying reason for this is that for each $x$ we have $\E_\Sigma R_\theta f(x)=f(x)$. That is, $\E_\Sigma R_\theta f=f$. From this it follows that if $f,g$ and $h$ are functions defined on $G$, then $\E_\Sigma R_\theta f*R_\theta g*R_\theta h=f*g*h$. (Here we are using independence, so that expectations of products are products of expectations.) We even have that
\[\E_\Sigma\prod_{i=1}^k R_\theta f_i*R_\theta g_i*R_\theta h_i =\prod_{i=1}^kf_i*g_i*h_i\]
for any functions $f_1,\dots,f_k,g_1,\dots,g_k,h_1,\dots,h_k$. 

This implies that a product of basic anti-uniform functions is a convex combination of products of functions built like basic anti-uniform functions but out of random restrictions instead. Since the random restrictions have smaller support, there are fewer of them -- or rather, they can be approximated by a smaller net -- which gives us the small set we are looking for.

Unfortunately, when we have $0\leq f\leq 1$ rather than $0\leq f\leq\mu$, the number of random restrictions is too large for our argument to work. However, in this case (which we would expect in advance to be easier) there is something else we can do. We first choose three independent random sets $A,B,C\subset G$ of size $qn$. 
Given a function $f$ with $0\leq |f|\leq 1$, we then define $R_Af$ as follows. We choose a random $u\in G$ and we then set $R_Af(x)=q^{-1}f(x)$ if $x\in A+u$ and 0 otherwise. We make similar definitions for $R_Bf$ and $R_Cf$. We have that $\E_\Sigma R_Af=f$, where here $A$ is fixed and $\E_\Sigma$ is the average over the different $R_A f$ corresponding to different choices of $u$, and similarly for $B$ and $C$.
With the same convention that different random restrictions are chosen independently, we also have that 
\[\E_\Sigma\prod_{i=1}^k R_A f_i*R_B g_i*R_C h_i =\prod_{i=1}^kf_i*g_i*h_i\]
for any functions $f_1,\dots,f_k,g_1,\dots,g_k,h_1,\dots,h_k$. More generally, let us take $Rf$ to be $R_\theta f$ if $0\leq f\leq\mu$ and $R_Af$, $R_Bf$ or $R_Cf$ if $0\leq f\leq 1$, depending on whether $f$ appears first, second or third in the convolution. Then we have the identity
\[\E_\Sigma\prod_{i=1}^k R f_i*R g_i*R h_i =\prod_{i=1}^kf_i*g_i*h_i\]
regardless of which kinds of functions the $f_i$, $g_i$ and $h_i$ are.

The reason this observation does not instantly prove what we want is that in order to prove that $\langle\mu-1,\phi\rangle$ is small for every $\phi\in\cf$ we shall need the functions in $\cf$ to be bounded. This is almost always true when we take convolutions of random restrictions, but not quite always. So we need to prove that the functions that are not bounded form a sufficiently small proportion of the functions in $\cf$ that we can remove them from the convex combination above and still have an approximate equality, where the approximation is in the uniform norm. This we shall do with the help of the following definition and lemma.

\begin{definition} Let $G$ be an Abelian group of order $n$, let $q\in[0,1]$ and let $W_1,W_2,W_3$ be subsets of $G$. Then the triple $(W_1,W_2,W_3)$ is $(\e,q)$-\emph{good} if it has the following properties:
\begin{enumerate} 
\item Each $W_i$ has size $(1+o(1))qn$.
\item Let $V_1,V_2,V_3$ be independent random subsets of $G$, where each element of each set is chosen with probability $q$. For each $i$ and each $u_i \in G$, let $\omega_{i,u_i}=q^{-1}\chi_{W_i + u_i}$ and $\nu_i=q^{-1}\chi_{V_i}$. Then, with probability $1-o(1)$, all convolutions $\b_1*\b_2*\b_3^-$ where each $\b_i$ is either $\omega_{i,u_i}$ or $\nu_i$ have the property that every value they take lies between $1-\e$ and $1+\e$. 
\end{enumerate}
We shall say that the triple $(V_1,V_2,V_3)$ \emph{complements} the triple $(W_1,W_2,W_3)$ if $V_1, V_2, V_3$ satisfy the conclusion of condition 2.
\end{definition}

\begin{lemma} \label{threesets}
For every $0<\e\leq 1$ there exists a constant $D$ with the following property. Let $G$ be an Abelian group of order $n$ and let $W_1,W_2$ and $W_3$ be independent random subsets of $G$ with elements chosen with probability $q$, where $q=Dn^{-2/3}(\log n)^{1/3}$. Then the triple $(W_1,W_2,W_3)$ is $(\e,q)$-good with probability at least $1-o(1)$.
\end{lemma}

\begin{proof}
Let $D$ be the constant given by Lemma \ref{tripleconvolution} (where it is called $C$), let $(V_1,V_2,V_3)$ be chosen according to the same distribution as $(W_1,W_2,W_3)$ and let $\b_1,\b_2$ and $\b_3$ be one of the possible choices for the $\b_i$ as in the definition above. Then, by Lemma \ref{tripleconvolution}, the probability that the values of $\b_1*\b_2*\b_3^-$ all lie between $1-\e$ and $1+\e$ is $1-o(n^{-3})$. Therefore, the probability is $1-o(1)$ that this is true for all $(n+1)^3$ possible choices of $\b_1,\b_2,\b_3$.
\end{proof}

Now let us define a set of bounded functions $\Psi_1$ and prove that every product of at most $k$ basic anti-uniform functions can be approximated by a convex combination of functions in $\Psi_1$. We will then find a fairly small subset $\Psi\subset\Psi_1$ with the same property.

Let $q$ be as in Lemma \ref{threesets} and let $(W_1,W_2,W_3)$ be a good triple. Let $p=Cn^{-2/3}(\log n)^{1/3}=Cq/D$ and let $U_1,U_2,U_3$ be independent sets where each element is chosen independently with probability $p$. Let $\mu_i=p^{-1}\chi_{U_i}$ for each $i$. 


Given any function $f:G\to\C$, write $\supp(f)$ for the support of $f$ and define $\sigma(f)$ to be the function that takes the value $q^{-1}$ on $\supp(f)$ and $0$ elsewhere. This is a kind of ``normalized support" of $f$. Given a function $h$, let $h^*$ be the function given by $h^*(x) = \ol{h(-x)}$. In particular, $\supp(h^*) = - \supp(h)$. We shall now let $\Psi_1$ consist of all functions of the form $\prod_{i=1}^kf_i*g_i*h_i^*$ with the following properties:
\begin{itemize}
\item For each $i$, $\supp(f_i)$ is contained in either $U_1$ or a translate of $W_1$, $\supp(g_i)$ is contained in either $U_2$ or a translate of $W_2$, and $\supp(h_i)$ is contained in either $U_3$ or a translate of $W_3$.
\item For each $i$, $\|f_i\|_\infty$, $\|g_i\|_\infty$ and $\|h_i\|_\infty$ are all at most $q^{-1}$.
\item For each $i$, all of $\supp(f_i)$, $\supp(g_i)$ and $\supp(h_i)$ have size at most $2qn$.
\item $\Bigl\|\prod_{i=1}^k\sigma(f_i)*\sigma(g_i)*\sigma(h_i^*)\Bigr\|_\infty\leq 3/2$.
\end{itemize}

We also need to change slightly the notion of a basic anti-uniform function. We now define it to be a convolution $u_1*u_2*u_3^*$ such that for each $i$ we either have $0\leq|u_i(x)|\leq\mu_i(x)$ for every $x$ or we have $0\leq|u_i(x)|\leq 1$ for every~$x$. (The earlier definition had $U_1=U_2=U_3$.)

\begin{corollary} \label{uniformbound}
Suppose that $(W_1, W_2, W_3)$ is $(\eps, q)$-good, with $q = Dn^{-2/3}(\log n)^{1/3}$ and $\epsilon = 1/4k$, and let $U_1,U_2,U_3$ and $\Psi_1$ be as defined above. Then, with probability $1-o(1)$, every product of at most $k$ basic anti-uniform functions can be approximated up to $\eta$ in the uniform norm by a convex combination of functions in $\Psi_1$.
\end{corollary}

\begin{proof}
Since $(W_1,W_2,W_3)$ is a good triple, the probability that a random triple $(V_1,V_2,V_3)$, where the elements of each $V_i$ are chosen independently with probability $q$, complements $(W_1,W_2,W_3)$ is $1-o(1)$.
But we can choose the triple $(V_1,V_2,V_3)$ by first choosing a triple $(U_1,U_2,U_3)$, with elements chosen independently with probability $p$, and then letting each $V_i$ be a random subset of $U_i$ with elements chosen independently with probability $q/p=D/C$. Therefore, with probability $1-o(1)$ the triple $(U_1,U_2,U_3)$ is such that with probability $1-o(1)$ a random triple of subsets $(V_1,V_2,V_3)$ chosen in this way complements the triple $(W_1,W_2,W_3)$. In particular, this also implies that, with probability $1-o(1)$, all convolutions $\b_1*\b_2*\b_3^-$, where each $\b_i$ is either $\omega_{i,u_i}$ or $\mu_i = p^{-1} \chi_{U_i}$, satisfy $\|\b_1*\b_2*\b_3^-\|_\infty \leq 2$.

Now let us fix $(U_1,U_2,U_3)$ such that this is the case. Let $\prod_{i=1}^kf_i*g_i*h_i^*$ be a product of basic anti-uniform functions, where each $f_i$ either satisfies $0\leq f_i\leq 1$ or $0\leq f_i\leq\mu_1$, and similarly for $g_i$ and $h_i$ with $\mu_2$ and $\mu_3$. 
If $0\leq f_i\leq 1$, then let the random restriction $Rf_i$ be defined as follows. We choose $u$ uniformly at random from $G$ and then set $Rf_i(x)$ to be $q^{-1}f_i(x)$ if $x\in W_1+u$ and $0$ otherwise. If $0\leq f_i\leq\mu_1$, then let $Rf_i(x)=(p/q)f_i(x)$ with probability $q/p$ and $0$ otherwise, with the choices being independent. Define the random restrictions $Rg_i(x)$ and $Rh_i(x)$ in the obvious corresponding ways.

As remarked earlier, we then have that
\[\prod_{i=1}^kf_i*g_i*h_i^*=\E_\Sigma \prod_{i=1}^kRf_i*Rg_i*Rh_i^*.\]
This does not yet finish the proof, since the functions $\prod_{i=1}^kRf_i*Rg_i*Rh_i^*$ do not necessarily belong to $\Psi_1$. However, because $(W_1,W_2,W_3)$ is a good triple, we see that with probability $1-o(1)$ we have $\|\sigma(Rf_i)*\sigma(Rg_i)*\sigma(Rh_i^*)\|_\infty\leq 1+\e$ for every $i$, which, since $(1+\e)^k\leq 3/2$, implies that $F =\prod_{i=1}^kRf_i*Rg_i*Rh_i^*$ belongs to $\Psi_1$.

Let $\Sigma_1$ be the subset of the probability space $\Sigma$ for which the random function $F$ belongs to $\Psi_1$ and let $\Sigma_2=\Sigma\setminus\Sigma_1$. We know that $\|Rf_i*Rg_i*Rh_i^*\|_\infty \leq(p/q)^3\|\beta_1*\beta_2*\beta_3^-\|_\infty$, where each $\beta_i$ is equal to $\mu_i$ or $\omega_{i,u_i}$, so $\|Rf_i*Rg_i*Rh_i^*\|_\infty \leq 2(p/q)^3=2(C/D)^3$. Therefore, we always have the bound $\|F\|_\infty\leq(2C/D)^{3k}$. 

By the law of total probability, we have
\[ \prod_{i=1}^kf_i*g_i*h_i^*=\P[F\in\Psi_1]\,\E_{\Sigma_1}F+\P[F\notin\Psi_1]\,\E_{\Sigma_2}F.\]
The first term is a convex combination of functions in $\Psi_1$ and the second has $\ell_\infty$ norm $o(1)(2C/D)^{3k}=o(1)$. This proves the result.
\end{proof}

It remains to show that $\Psi_1$ has a subset $\Psi$ that is not too large, such that every convex combination of functions in $\Psi_1$ can be uniformly approximated by a convex combination of functions in $\Psi$. This we do in a crude way. Given $\d>0$, let $\D$ be a $\d$-net of the unit disc in $\C$ that includes 0 and has size at most $16/\d^2$ and let $\Psi$ consist of all functions $\prod_{i=1}^kf_i*g_i*h_i^*\in\Psi_1$ such that every value of every $f_i,g_i$ and $h_i$ is of the form $q^{-1}z$ for some $z\in\D$. 

Now let $\prod_{i=1}^kf_i*g_i*h_i^*\in\Psi_1$. For each $i$ we can choose functions $f_i',g_i'$ and $h_i'$ taking values in $q^{-1}\D$ with $|f_i-f_i'|\leq\d\sigma(f_i)$, $|g_i-g_i'|\leq\d\sigma(g_i)$ and $|h_i-h_i'|\leq\d\sigma(h_i)$. By telescoping, it follows from the triangle inequality that 
\[\|f_i*g_i*h_i^*-f_i'*g_i'*h_i'^*\|_\infty\leq 3\d\|\sigma(f_i)*\sigma(g_i)*\sigma(h_i^*)\|_\infty\]
and, more generally, that
\[\Bigl\|\prod_if_i*g_i*h_i^*-\prod_if_i'*g_i'*h_i'^*\Bigr\|_\infty
\leq 3k\d\Bigl\|\prod_i\sigma(f_i)*\sigma(g_i)*\sigma(h_i^*)\Bigr\|_\infty\leq 9k\d/2.\]

We are now ready to prove the main result of this section.

\begin{lemma} \label{maintransfer}
For every $\e>0$ and every positive integer $k$ there exist constants $C$ and $D$ with the following property. Let $p=Cn^{-2/3}(\log n)^{1/3}$ and let $q=Dp/C=Dn^{-2/3}(\log n)^{1/3}$. Let $U$ be a random set where each element is chosen independently with probability $p$. Let $\mu=p^{-1}\chi(U)$. Let $U_1,U_2,U_3$ be further random sets chosen independently in the same way and for each $i$ let $\mu_i=p^{-1}\chi(U_i)$. Then, with probability $1-o(1)$, $|\langle\mu-1,\xi\rangle|\leq\e$ for every product $\xi$ of at most $k$ basic anti-uniform functions (as defined just before Corollary \ref{uniformbound}).
\end{lemma}

\begin{proof}
By Corollary \ref{uniformbound}, with probability $1-o(1)$ we can express each $\xi$ as a convex combination of functions in $\Psi_1$ plus an error term that is uniformly bounded by $\e/4$. And then, by the remarks above, if we set $\d=\e/18k$, we can express every convex combination of functions in $\Psi_1$ by a convex combination of functions in $\Psi$ plus an error term that is again uniformly bounded by $\e/4$. Thus, $\xi=\xi_1+\xi_2$, where $\xi_1$ is a convex combination of functions in $\Psi$ and $\|\xi_2\|_\infty \leq\e/2$.

Since 
\[|\langle\mu-1,\xi\rangle|\leq|\langle\mu-1,\xi_1\rangle|+|\langle\mu-1,\xi_2\rangle|\leq|\langle\mu-1,\xi_1\rangle|+p^{-1}|U|\e/2n,\]
and since $|U|\leq 3pn/2$ with probability $1-o(1)$, it is enough to prove that with probability $1-o(1)$, $|\langle\mu-1,\psi\rangle|\leq\e/4$ for every $\psi\in\Psi$. For this we can use a union bound. That is, we need to obtain upper bounds for the number of functions in $\Psi$ and for the probability that $|\langle\mu-1,\psi\rangle|>\e/4$ for any individual $\psi\in\Psi$.

For any given set $A$ of size at most $2qn$, the number of functions supported in $A$ and taking values in $q^{-1}\Delta$ is $|\Delta|^{2qn}$. Since we are taking $\d$ to be $\e/18k$, this is at most $(10^4k^2/\e^2)^{2qn}$. The number of possible supports for each function involved in a product of convolutions in $\Psi$ is at most $n+\binom{2pn}{2qn}$, which is at most $(4C/D)^{2qn}$. Therefore, the number of functions in $\Psi$ is at most $(10^5k^2C/\e^2D)^{6kqn}$. Since each $\psi\in\Psi$ has $\ell_\infty$ norm at most 2, the probability that $|\langle\mu-1,\psi\rangle| > \e/4$ is, by Bernstein's inequality, at most $2\exp(-\e^2n^2/32(4p^{-1}n+\frac 16p^{-1}\e n))\leq2\exp(-\frac 1{160}\e^2pn)$, since $\langle\mu-1,\psi\rangle$ is an average of $n$ random variables, each of mean zero, second moment at most $4p^{-1}$, and maximum at most $2p^{-1}$. 

We are therefore done provided that $(10^5k^2C/\e^2D)^{6kqn}\exp(-\frac 1{160}\e^2pn)=o(1)$. Taking logs, we require
\[\frac 1{160}\e^2pn-6kqn\log(10^5k^2C/\e^2D)\]
to tend to infinity, for which it is enough if $C/D > 10^3 k \e^{-2}\log(10^5k^2C/\e^2D)$. If we let $D$ be the constant given by Lemma \ref{threesets}, then we are done.
\end{proof}

We are not quite done, since we have assumed that our basic anti-uniform functions were built from functions supported on random sets that were independent of $U$. To recover the statement for basic anti-uniform functions built from functions supported on $U$ itself, we think of $U$ as being the union of $t$ independent random sets $U_1, \dots, U_t$, each chosen with probability $\frac{1}{t} C n^{-2/3}(\log n)^{1/3}$, with $U_i$ having associated measure $\mu_i$. With high probability, these sets are all disjoint, so $\mu = \E_i \mu_i$ and any function $f$ with $0 \leq |f| \leq \mu$ can be written as an expectation $\E_i f_i$ where $0 \leq |f_i| \leq \mu_i$ for each $i$. In particular, any expression of the form $\sp{\mu - 1, \xi}$ can be rewritten as the expectation over expressions of the form $\sp{\mu_i - 1, \xi_{j_1, \dots, j_{3k}}}$, where the indices $j_1, \dots, j_{3k}$ indicate which of the sets $U_1, \dots, U_t$ the basic anti-uniform function $\xi_{j_1, \dots, j_{3k}}$ is built over.

Since $i, j_1, \dots, j_{3k}$ all vary over $1, 2, \dots, t$, the probability that any two of them are equal is at most $\binom{3k+1}{2}/t$. If they are all different, Lemma~\ref{maintransfer} applied with $C/t$ and $\epsilon/2$ implies that $|\sp{\mu_i - 1, \xi_{j_1, \dots, j_{3k}}}| \leq \epsilon/2$, while in general, using that $\|\mu_a*\mu_b*\mu_c^-\|_\infty \leq 2$ with high probability for all $a, b, c$, we have the bound $|\sp{\mu_i - 1, \xi_{j_1, \dots, j_{3k}}}| \leq 2^{k+1}$. Therefore,
\begin{align*}
|\sp{\mu - 1, \xi}| & =| \E_{i, j_1, \dots, j_{3k}} \sp{\mu_i - 1, \xi_{j_1, \dots, j_{3k}}}|\\ 
& \leq \frac{\epsilon}{2} + 2^{k+1}\frac{\binom{3k+1}{2}}{t} \leq \epsilon,
\end{align*}
for an appropriate $t$. This completes the proof of our transference principle. In the next section, we will show how to apply it.

\section{Freiman homomorphisms are affine almost everywhere}

In this section, we shall use the main result of the previous section to prove that if $G$ is an Abelian group of order $n$ and $U$ is a random subset of $G$ chosen with probability $C(\log n)^{1/3} n^{-2/3}$, then, with high probability, for every finitely generated Abelian group $H$ and every Freiman homomorphism $\phi:U\to H$ there is an affine homomorphism $\psi:G\to H$ such that $\psi(u)=\phi(u)$ for at least 99.99\% of the elements $u\in U$. Then in the next section we shall get from 99.99\% to 100\%.

The condition here that $H$ is finitely generated is not important, since we can always restrict to the subgroup generated by the image of $U$, but it will be convenient for us to assume it so that we can take expectations over the characters of $H$. 

Suppose now that $\phi:U\to H$ is a Freiman homomorphism. To apply our transference principle, we need a function from $U$ to $\C$, and the obvious way of producing such a function is to compose $\phi$ with a character, except that we shall normalize this function by multiplying it by the characteristic measure $\mu$ of $U$. Accordingly, if $\chi$ is a character on $H$, let us define $f_\chi$ to be the function $\mu(\chi\circ\phi)$. Note that $\chi\circ\phi$ is a Freiman homomorphism from $U$ to the unit circle. 

Let us now fix $\chi$ and write $f$ for $f_\chi$. Let $\eta>0$ be a constant to be chosen later. By the transference principle, we can find a function $g:G\to\C$ such that $\|g\|_\infty\leq 1$, $|M(g)-M(f)|\leq\eta$, and $|\sp{g-f,\tau}|\leq\eta$ for every character $\tau:G\to\C$. 

We shall now show that $M(f)\approx 1$, which implies that $M(g)\approx 1$, which implies that $g$ is approximately equal to a character $\tau$, which implies that $\sp{f,\tau}\approx 1$. This will show that $\|\hf_\chi\|_{12}\approx 1$ for every character $\chi:H\to\C$, which (as later calculations will reveal) effectively allows us to prove that $\phi$ coincides with a homomorphism of order 6 on almost all of $U$. But since $3U=G$ with high probability, a homomorphism of order 6 on almost all of $U$ gives us a homomorphism of order 2 on almost all of $G$. 

In the rest of this section we shall give the details. To avoid a lot of repetition, let us establish once and for all that $G$ is a finite Abelian group of order $n$, $\eta>0$ is a small positive constant, $U$ is a random subset of $G$ chosen with probability $C(\log n)^{1/3} n^{-2/3}$, and $\mu$ is the characteristic measure of $U$. In addition, we shall adopt the following convention. When we make a statement that involves $U$, it is to be understood that that statement is valid with probability $1-o(1)$. Moreover, if that statement involves a universal quantifier (either explicitly, or via the word ``let"), it is to be understood that the ``with probability $1-o(1)$" comes \textit{before} the quantifier. For instance, if we write, ``Let $A$ be such that $P(A)$. Then $Q(A,U)$," that is shorthand for, ``With probability $1-o(1)$, $Q(A,U)$ for every $A$ such that $P(A)$." 

We begin with a simple probabilistic lemma. Though it has a much easier proof, we note here that it follows as a corollary of the fact that $\sp{\mu - 1, \xi}$ is small for all basic anti-uniform functions $\xi$.

\begin{lemma}\label{Mofmu}
$M(\mu)\geq 1-\eta$.
\end{lemma}

Recall that for each character $\tau : G \rightarrow \C$ and any function $f : G \rightarrow \C$, the Fourier transform is given by $\hf(\tau) = \sp{f, \tau} = \E_x f(x) \ol{\tau(x)}$. We also have the Fourier inversion formula, $f(x) = \sum_\tau \hf(\tau) \tau(x)$. In keeping with these normalisations, we write $\|f\|_p = (\E_x |f(x)|^p)^{1/p}$, while $\|\hf\|_p = (\sum_\tau |\hf(\tau)|^p)^{1/p}$. This allows us to state Parseval's formula $\|f\|_2 = \|\hf\|_2$ and a number of related identities in a clean fashion.

\begin{lemma}\label{ell12estimate}
Let $H$ be a finitely generated Abelian group, let $\phi:U\to H$ be a Freiman homomorphism, let $\chi:H\to\C$ be a character, and let $f=\mu(\chi\circ\phi)$. Then $\|\hf\|_{12}^{12}\geq 1-36\eta$.
\end{lemma}

\begin{proof}
As mentioned above, there exists a function $g:G\to\C$ such that $\|g\|_\infty\leq 1$, $M(g)\geq M(f)-\eta$, and $\sp{f,\tau}\geq\sp{g,\tau}-\eta$ for every character $\tau:G\to\C$. But 
\[M(f)=\E_{(x,y,z,w)\in\Gamma}\mu(x)\mu(y)\mu(z)\mu(w)\chi(\phi(x)+\phi(y)-\phi(z)-\phi(w))=M(\mu),\]
since $\phi$ is a Freiman homomorphism. Therefore, by Lemma \ref{Mofmu}, $M(f)\geq 1-\eta$. It follows that $M(g)\geq 1-2\eta$.

Now the sum of $g(x)g(y)\overline{g(z)g(w)}$ over degenerate additive quadruples $x+y=z+w$ (that is, ones where $x=z$ and $y=w$ or $x=w$ and $y=z$) is a non-negative real number. Therefore, $\|g\|_{U^2}^4\geq 1-2\eta$. Therefore, $\|\hg\|_4^4\geq 1-2\eta$. Since $\|\hg\|_2=\|g\|_2\leq 1$, it follows that $\|\hg\|_\infty^2\geq 1-2\eta$ and therefore that $\|\hg\|_\infty\geq 1-2\eta$. 

Since $\sp{f,\tau}\geq\sp{g,\tau}-\eta$ for every character $\tau:G\to\C$, it follows that $\|\hf\|_\infty\geq 1-3\eta$ and therefore that $\|\hf\|_{12}^{12}\geq 1-36\eta$, as claimed.
\end{proof}

The proof of the above lemma is the only place where we need to use the transference result of the previous section. (However, the lemma is crucial to our main argument.)

Let us now regard $H$ and $\phi$ as fixed (though it is important that the statement we are proving about $U$ is one that with high probability holds for all $H$ and $\phi$). We shall now convert this statement about $\ell_{12}$ norms of Fourier transforms into a statement about Freiman homomorphisms of order 6. To begin with, we obtain \textit{near} homomorphisms. The rough meaning of the next statement is that for almost every additive 12-tuple $(x_1,\dots,x_{12})$ in $U$, we also have $\phi(x_1)+\dots+\phi(x_{6})=\phi(x_7)+\dots+\phi(x_{12})$.

\begin{lemma} \label{almosthom}
Let $\Gamma$ be the set of all $(x_1,\dots,x_{12})\in G^{12}$ such that $x_1+\dots+x_6=x_7+\dots+x_{12}$ and let $\Phi$ be the set of all $(x_1,\dots,x_{12})\in\Gamma$ such that $\phi(x_1)+\dots+\phi(x_6)=\phi(x_7)+\dots+\phi(x_{12})$. Then
\[\E_{(x_1,\dots,x_{12})\in\Gamma}\ \mu(x_1)\dots\mu(x_{12})\mathbf{1}_{(x_1,\dots,x_{12})\in\Phi}\geq 1-36\eta.\]
\end{lemma}

\begin{proof}
Let us write $f_{\chi}$ for the function $\mu(\chi\circ\phi)$. (Previously we just wrote $f$, but now we need to consider all such functions.) Then, by Lemma \ref{ell12estimate}, we know that $\|\hf_\chi\|_{12}^{12}\geq 1-36\eta$ for every character $\chi$, and therefore that $\E_\chi\|\hf_\chi\|_{12}^{12}\geq 1-36\eta$. Note that
\[\|\hf_\chi\|_{12}^{12} = \sum_{\tau} |\hf(\tau)|^{12} = \sp{\hf^6, \hf^6} = \sp{f * \dots *f, f* \dots *f} = \E_{(x_1, \dots, x_{12}) \in \Gamma} f(x_1) \dots f(x_6) \ol{f(x_7) \dots f(x_{12})}\]
and, therefore, $\E_\chi\|\hf_\chi\|_{12}^{12}$ is equal to 
\[\E_\chi\E_{(x_1,\dots,x_{12})\in\Gamma}\ \mu(x_1)\dots\mu(x_{12})\chi(\phi(x_1)+\dots+\phi(x_6)-\phi(x_7)-\dots-\phi(x_{12})).\]
The expectation over $\chi$ is 1 if $(x_1,\dots,x_{12})\in\Phi$ and 0 otherwise, so we have precisely the expectation on the left-hand side of the inequality we are trying to prove.
\end{proof}

The rough idea of what we want to do next is to define a function $\psi:3U\to H$ by taking $\psi(u)$ to be the most popular value of $\phi(x_1)+\phi(x_2)+\phi(x_3)$ when $x_1+x_2+x_3=u$. Lemma \ref{almosthom} implies that most of the time one value is predominant and that the function $\psi$ thus defined has the property that $u_1+u_2=u_3+u_4$ almost always implies that $\psi(u_1)+\psi(u_2)=\psi(u_3)+\psi(u_4)$. This in turn implies that $\psi$ agrees almost everywhere with an affine homomorphism.

However, arguments of the above kind can get quite messy, so for the sake of tidiness we shall instead postpone for as long as we can the moment where we have to use an averaging argument to commit ourselves to a particular function. The price we pay for this is that we must consider functions whose values are probability distributions on $H$ rather than single elements of $H$. However, this is a very natural thing to do: instead of defining $\psi(u)$ to be the most popular value of $\phi(x_1)+\phi(x_2)+\phi(x_3)$ such that $x_1+x_2+x_3=u$, we define it to be something like the probability distribution where the probability that $\psi(u)=h$ is the probability that $\phi(x_1)+\phi(x_2)+\phi(x_3)=h$ given that $x_1+x_2+x_3=u$. This is not quite accurate, because we need to take account of the fact that $\mu*\mu*\mu$ is not quite constant, so here is the precise definition.

\begin{definition} 
Let $\phi:U\to H$ and let $\pi(H)$ be the set of all non-negative real-valued finitely supported functions on $H$. Then, for each $x\in G$, let $\psi(x)\in\pi(H)$ be the function defined by the formula
\[\psi(x)(h)=\E\{\mu(x_1)\mu(x_2)\mu(x_3):x_1+x_2+x_3=x, \phi(x_1)+\phi(x_2)+\phi(x_3)=h\}.\]
\end{definition}
Typically, $\psi(x)(h)$ will be almost 1 for one $h\in H$ and the sum over $h$ will be asymptotically equal to 1. In other words, it will be concentrated at one $h$, so we can think of $\psi$ as like a function from $G$ to $H$ but slightly fuzzy.

We also need to be able to make sense of expressions such as $\psi(x)+\psi(y)$. That is easy enough: we simply convolve $\psi(x)$ with $\psi(y)$. To avoid confusion, we shall write $\psi(x)*\psi(y)$, which we define formally as follows. Note that the normalization is different from that used earlier: it is more appropriate for convolutions of probability distributions.

\begin{definition}
Let $p,q\in\pi(H)$. The \emph{convolution} $p*q$ of $p$ and $q$ is defined by the formula $(p*q)(h)=\sum_{h_1+h_2=h}p(h_1)q(h_2)$.
\end{definition}
If $\psi(x)$ is concentrated at $h_1$ and $\psi(y)$ is concentrated at $h_2$, then $\psi(x)*\psi(y)$ is concentrated (but not quite as strongly) at $h_1+h_2$. In this situation, convolution can be thought of as a fuzzy version of addition.

We would also like a fuzzy version of subtraction, so we define $p_-(h)$ to be $p(-h)$. Then the fuzzy analogue of $\psi(x)-\psi(y)$ is $\psi(x)*\psi_-(y)$.

Finally, we need a fuzzy version of equality: there are no circumstances under which we can reasonably expect two expressions such as $\psi(x_1)*\psi(x_2)$ and $\psi(x_3)*\psi(x_4)$ to be exactly equal, but we can certainly expect them to be roughly equal. To measure this, we define an inner product in an obvious way. Again, this is defined with a different normalization from earlier.

\begin{definition}
Let $p$ and $q$ be two non-negative finitely supported functions on $H$. Then their \emph{inner product} $\sp{p,q}$ is defined to be $\sum_{h\in H}p(h)q(h)$.
\end{definition}

\noindent If $p$ and $q$ both sum to 1 (or approximately 1), then we regard $p$ and $q$ as close if $\sp{p,q}$ is close to 1. This is a stronger statement than merely that $p$ and $q$ are similar functions: it implies also that $p$ and $q$ are both fairly concentrated at a single value. However, that is exactly what we want to show, so this notion of closeness is useful.
\medskip

The next lemma is just a reformulation of Lemma \ref{almosthom}.

\begin{lemma} \label{almosthom2}
$\E_{z_1+z_2=z_3+z_4}\sp{\psi(z_1)*\psi(z_2),\psi(z_3)*\psi(z_4)}\geq 1-36\eta$.
\end{lemma}

\begin{proof}
The inner product expands to
\[\sum_{h_1+h_2=h_3+h_4}\psi(z_1)(h_1)\psi(z_2)(h_2)\psi(z_3)(h_3)\psi(z_4)(h_4).\]
But $\psi(z_i)(h_i)$ can be written as
\[\E\{\mu(x_{3i-2})\mu(x_{3i-1})\mu(x_{3i}):x_{3i-2}+x_{3i-1}+x_{3i}=z_i,\phi(x_{3i-2})+\phi(x_{3i-1})+\phi(x_{3i})=h_i\}.\]
Therefore, taking the expectation of the inner product over all additive quadruples $z_1+z_2=z_3+z_4$, we obtain
\[\E\{\mu(x_1)\dots\mu(x_{12}):x_1+\dots+x_6=x_7+\dots+x_{12},\phi(x_1)+\dots+\phi(x_6)=\phi(x_7)+\dots+\phi(x_{12})\}\]
which is another way of writing $\E_{(x_1,\dots,x_{12})\in\Gamma}\ \mu(x_1)\dots\mu(x_{12})\mathbf{1}_{(x_1,\dots,x_{12})\in\Phi}$. Thus, the result follows from Lemma \ref{almosthom}.
\end{proof}

\begin{corollary} \label{almosthom3}
$\E_{z_1-z_2=z_3-z_4}\sp{\psi(z_1)*\psi_-(z_2),\psi(z_3)*\psi_-(z_4)}\geq 1-36\eta$.
\end{corollary}

\begin{proof}
The inner product expands to
\[\sum_{h_1+h_2=h_3+h_4}\psi(z_1)(h_1)\psi(z_2)(-h_2)\psi(z_3)(h_3)\psi(z_4)(-h_4)\]
which equals
\[\sum_{h_1-h_2=h_3-h_4}\psi(z_1)(h_1)\psi(z_2)(h_2)\psi(z_3)(h_3)\psi(z_4)(h_4)\]
which equals
\[\sum_{h_1+h_4=h_3+h_2}\psi(z_1)(h_1)\psi(z_2)(h_2)\psi(z_3)(h_3)\psi(z_4)(h_4)\]
which equals $\sp{\psi(z_1)*\psi(z_4),\psi(z_2)*\psi(z_3)}$. We are taking the expectation of this quantity over quadruples $(z_1,z_2,z_3,z_4)$ such that $z_1+z_4=z_2+z_3$, so the result follows from Lemma \ref{almosthom2}.
\end{proof}

We are about to define $\theta:G\to\pi(H)$ as the convolution of $\psi$ with $\psi_-^-$, where $\psi_-^-(x)(h)$ is defined to be $\psi(-x)(-h)$. However, we must first say what ``convolution" means here.

\begin{definition}
Let $\phi,\psi:G\to\pi(H)$. The \emph{convolution} $\phi*\psi:G\to\pi(H)$ is defined by the formula
\[(\phi*\psi)(x)=\E_{x_1+x_2=x}\phi(x_1)*\psi(x_2).\]
\end{definition}

Let $\zeta:G\to H$ be a Freiman homomorphism and let $\phi:G\to\pi(H)$ be defined by $\phi(x)(h)=1$ if $h=\zeta(x)$ and 0 otherwise. In that case, $\phi*\phi_-^-(x)(h)=1$ if for every $x_1-x_2=x$ we have $\zeta(x_1)-\zeta(x_2)=h$ and 0 otherwise. In other words, $\phi*\phi_-^-$ is essentially the well-defined function from $G-G$ to $H$ that is induced by the fact that $\zeta$ is a Freiman homomorphism. Lemma \ref{almosthom2} can be thought of as saying that $\psi$ is an \textit{approximate} homomorphism. We therefore expect $\theta=\psi*\psi_-^-$ to be ``approximately well-defined". We shall show that for every $x$ the function $\theta(x)$ is concentrated at some value $\gamma(x)\in H$, and that $\gamma$ is a group homomorphism.

First, we need a lemma that can be thought of as a kind of triangle inequality, with $1-\sp{p,q}$ being the ``distance" between $p$ and $q$. This distance is similar to the Ruzsa distance between two sets: in particular, a function need not be close to itself.

\begin{lemma} \label{triangleineq}
Let $p$, $q$ and $r$ be elements of $\pi(H)$, each of which sums to at most $1+\beta$, where $\beta \leq 1$. Then 
\[1-\sp{p,r}\leq 1-\sp{p,q}+1-\sp{q,r}+3\beta.\]
\end{lemma}

\begin{proof}
For every $h\in H$ we have the inequality $(1+\beta-p(h))(1+\beta-r(h))\geq 0$, since $p$ and $r$ take values in $[0,1+\beta]$. It follows that $p(h)r(h)\geq (1+\beta)(p(h)+r(h))-(1+\beta)^2$. Therefore, since $q$ also takes values in $[0,1+\beta]$,
\begin{eqnarray*}
1-\sp{p,r}&=&1-\sum_hp(h)r(h)\\
&\leq& 1-(1+\beta)^{-1}\sum_h q(h)p(h)r(h)\\
&\leq& 1-\sum_h q(h)(p(h)+r(h)-(1+\beta))\\
&=&1-\sp{p,q}-\sp{q,r}+ (1 + \beta) \sum_h q(h)\\
&\leq& 1- \sp{p,q} + 1 - \sp{q,r} + 3 \beta,
\end{eqnarray*}
as claimed.
\end{proof}

Let us write $d(p,q)$ for $1-\sp{p,q}$. Then Lemma \ref{triangleineq} tells us that $d(p,r)\leq d(p,q)+d(q,r)+\beta$. Note that $d(p,q)$ can be negative, but it cannot be smaller than $-2\beta-\beta^2$. The fact that it can be negative turns out not to matter. 

Now we need an estimate that can be used to give us a $\beta$ to use in the previous lemma.

\begin{lemma} \label{almostconstant}
$\|\mu*\mu*\mu-1\|_\infty\leq\eta$. 
\end{lemma}

\begin{proof}
It will suffice to show that for any fixed $x$, the number of distinct ways $S$ of writing $x$ as a sum of three elements in $U$ satisfies $|S - \E S| \leq \eta p^3 n^2$ with probability $1 - o(1/n)$. A straightforward application of Janson's inequality (see~\cite{AS}) gives the required estimate for the probability that $S < \E S - \eta p^3 n^2$. We will therefore focus on estimating the probability that $S > \E S + \eta p^3 n^2$.

We will use the method described in Section 2.3.4 of~\cite{JR02}. Let $S'$ be the random variable counting the maximum number of disjoint three element sets each of which sum to $x$. We claim that $S \leq S' + 28$ with probability $1 - o(1/n)$. To prove the claim, let $\mathcal{S}_i$ be the collection of $i$ element subsets of $U$ giving rise to a triple summing to $x$. Form a graph $J$ whose vertices are the elements of $\mathcal{S}_3$, where two elements are joined if and only if they intersect. Then, with probability $1 - o(1/n)$, it is straightforward to verify that the maximum degree of $J$ is at most $4$ and the largest induced matching has size at most $3$. Note that $S'$ is the order of the largest independent set in $J$. Since at most $24$ vertices are joined to a maximal induced matching and the remaining set is independent, we have
\[S' \geq |J| - 24 = |\mathcal{S}_3| - 24.\]
Since it is also easy to verify that $|\mathcal{S}_2| \leq 3$ with probability $1 - o(1/n)$ and $\mathcal{S}_1 \leq 1$ (unless $G$ has characteristic $3$ and $x = 0$), the claim follows.


The required conclusion now follows from the inequality
\[\P[S' \geq \E S + t] \leq \exp\left(\frac{-t^2}{2(\E S + t/3)}\right),\]
which is Lemma 2 of~\cite{JR02}.
\end{proof}

\noindent \textbf{Remark.} It is possible to prove the above lemma in a slightly more elementary way, by deducing it from Lemma \ref{tripleconvolution}. To do this, one must split $\mu$ into several independent parts and use the fact that most of the terms that arise involve different parts. (We used a similar idea at the end of the previous section.)
\medskip

There is a simple way of measuring the well-definedness of $\theta$: we look at how small the distances $d(\theta(x),\theta(x))$ are.

\begin{lemma}\label{welldefined}
For every $x\in G$, $d(\theta(x),\theta(x))\leq 75\eta$.
\end{lemma}

\begin{proof}
Fix $x\in G$. We need to show that
\[\E_{z_1-z_2=z_3-z_4=x}\sp{\psi(z_1)*\psi_-(z_2),\psi(z_3)*\psi_-(z_4)}\geq 1-75\eta.\]
From Corollary \ref{almosthom3}, we know both that
\[\E_{w_1-w_2=w_3-w_4}(1-\sp{\psi(w_1)*\psi_-(w_2),\psi(w_3)*\psi_-(w_4)})\leq 36\eta\]
and that
\[\E_{w_1-w_2=w_3-w_4}(1-\sp{\psi(w_1+x)*\psi_-(w_2+x),\psi(w_3)*\psi_-(w_4)})\leq 36\eta.\]
Now $\sum_{h\in H}\psi(x)(h)=\mu*\mu*\mu(x)$, which is at most $1+\eta$, by Lemma \ref{almostconstant}. It follows from Lemma \ref{triangleineq} that
\[\E_{w_1,w_2}(1-\sp{\psi(w_1)*\psi_-(w_2),\psi(w_1+x)*\psi_-(w_2+x)})\leq 75\eta.\]
But 
\[\sp{\psi(w_1)*\psi_-(w_2),\psi(w_1+x)*\psi_-(w_2+x)}=\sp{\psi(w_1+x)*\psi_-(w_1),\psi(w_2+x)*\psi_-(w_2)}.\]
Therefore,
\[\E_{w_1,w_2}\sp{\psi(w_1+x)*\psi_-(w_1),\psi(w_2+x)*\psi_-(w_2)}\geq 1-75\eta,\]
which is (a slightly rewritten version of) what we needed to show.
\end{proof}

\begin{corollary}\label{thetaconcentrated}
For every $x$, there exists $h$ such that $\theta(x)(h)\geq 1-77\eta$.
\end{corollary}

\begin{proof}
We know that $\sum_h\theta(x)(h)^2\geq 1-75\eta$. Also, since $\sum_h\psi(y)(h)\leq 1+\eta$ for every $y$, $\sum_h\theta(x)(h)\leq\E_{x_1-x_2=x}\sum_{h_1, h_2}\psi(x_1)(h_1)\psi(x_2)(-h_2)$ is at most $(1+\eta)^2$. It follows that $\max_h\theta(x)(h)\geq(1-75\eta)(1+\eta)^{-2}\geq 1-77\eta$, as claimed.
\end{proof}

\begin{corollary}\label{thetaofnought}
$\theta(0)(0)\geq 1-77\eta$.
\end{corollary}

\begin{proof}
For every $h$, 
\[\theta(0)(h)=\E_x\sum_{h_1-h_2=h}\psi(x)(h_1)\psi(x)(h_2).\]
Let $h'$ be such that $\psi(x)(h')$ is the largest value over all $h$ of $\psi(x)(h)$.
Then, if $h \neq 0$, 
\[\sum_{h_1-h_2=h}\psi(x)(h_1)\psi(x)(h_2) \leq 2 \psi(x)(h') \sum_{h_1 \neq h'} \psi(x)(h_1) \leq \frac{(1 + \eta)^2}{2},\]
where we used that $\psi(x)(h') \sum_{h_1 \neq h'} \psi(x)(h_1)$ is bounded by an expression of the form $x((1+ \eta) - x) \leq (1+ \eta)^2/4$.
Therefore, if $h\ne 0$, $\theta(0)(h)\leq (1+\eta)^2/2$. Since this is less than $1-77\eta$ for $\eta$ sufficiently small, the only way that Corollary \ref{thetaconcentrated} can be true is if $\theta(0)(0)\geq 1-77\eta$.
\end{proof}

\begin{lemma}\label{closefunctions}
Let $\beta,\gamma\geq 0$ and let $p,q\in\pi(H)$ with $\sum_hp(h)$ and $\sum_hq(h)$ at most $1+\beta$. Suppose that $d(p,q)\leq\gamma$. Then there exists $h$ such that $p(h)q(h)\geq(1-\beta-\gamma)^2$.
\end{lemma}

\begin{proof}
We know that 
\[\sum_hp(h)q(h)\leq(\max_hp(h)^{1/2}q(h)^{1/2})\sum_hp(h)^{1/2}q(h)^{1/2}.\]
If the result is false, then $\max_hp(h)^{1/2}q(h)^{1/2}$ is less than $1-\beta-\gamma$, and, by the Cauchy--Schwarz inequality and the assumptions on $p$ and $q$, the inner sum is at most $1+\beta$. It follows that $\sp{p,q}<(1-\beta-\gamma)(1+\beta)\leq 1-\gamma$, and therefore that $d(p,q)>\gamma$, a contradiction.
\end{proof}

The next lemma tells us that inner products are ``approximately Lipschitz" functions of their arguments.

\begin{lemma}\label{continuity}
Let $0\leq\beta\leq 1/5$ and let $p,q,r\in\pi(H)$ be such that $\sum_hp(h)$, $\sum_hq(h)$ and $\sum_hr(h)$ are all at most $1+\beta$. Then $|\sp{p,r}-\sp{q,r}|\leq 5d(p,q)+10\beta$.
\end{lemma}

\begin{proof}
For any element $p\in\pi(H)$ and any $s\in[1,\infty)$, write $\|p\|_s$ for $\bigl(\sum_{h\in H}p(h)^s\bigr)^{1/s}$ and $\|p\|_\infty$ for $\max_hp(h)$. Then
\begin{eqnarray*}
|\sp{p,r}-\sp{q,r}|&=&|\sp{p-q,r}|\\
&\leq&\|p-q\|_1\|r\|_\infty\\
&\leq&(1+\beta)\|p-q\|_1.
\end{eqnarray*}
Let us write $\gamma$ for $d(p,q)$. By Lemma \ref{closefunctions}, there exists $h$ such that $p(h)q(h)\geq(1-\beta-\gamma)^2$, which implies that $p(h)+q(h)\geq 2(1-\beta-\gamma)$. Therefore, $\sum_{h'\ne h}(p(h')+q(h'))\leq 4\beta+2\gamma$. Also, $|p(h)-q(h)|$ is at most $1+\beta-(1-\beta-\gamma)^2/(1+\beta)\leq 4\beta+2\gamma$. Therefore, $\|p-q\|_1\leq 8\beta+4\gamma$. Since $\beta\leq 1/5$, this gives us the desired estimate.
\end{proof}

We would also like to know that convolutions are approximately Lipschitz.

\begin{corollary} \label{lipschitzconvolutions}
Let $0\leq\beta\leq1/2$ and let $p,q,r,s\in\pi(H)$ be such that $\|p\|_1,\|q\|_1,\|r\|_1,\|s\|_1\leq 1+\beta$. Then 
\[d(p*r,q*s)\leq (1+\beta)^2(d(p,q)+d(r,s))+8\beta^2.\]
\end{corollary}

\begin{proof}
We shall use the fact that $f*g_-(0)=\sp{f,g}$ for any two functions $f,g\in\pi(H)$. That implies that
\begin{eqnarray*}
\sp{p*r,q*s}&=&\sp{p*q_-,s*r_-}\\
&\geq&\sp{p,q}\sp{r,s}.
\end{eqnarray*}
Since $((1+\beta)^2-\sp{p,q})((1+\beta)^2-\sp{r,s})\geq 0$, 
\[\sp{p,q}\sp{r,s}\geq (1+\beta)^2(\sp{p,q}+\sp{r,s})-(1+\beta)^4.\]
The result follows after a quick calculation.
\end{proof}

The next lemma tells us that $\theta$ is close to a group homomorphism. Here and in what follows, we write $\delta_a$ for the function taking value $1$ at $a$ and $0$ everywhere else.

\begin{lemma}\label{homomorphism}
For every $x_1,x_2\in G$, $d(\theta(x_1+x_2),\theta(x_1)*\theta(x_2))\leq 1100\eta$.
\end{lemma}

\begin{proof}
By definition,
\[\theta(x_1+x_2)=\E_x\psi(x+x_1+x_2)*\psi_-(x).\]
We would like to begin by ``adding and subtracting $\psi(x+x_1)$''. More precisely, we would like to approximate $\theta(x_1+x_2)$ by
\begin{equation*}\label{eqn0}
\E_x\psi(x+x_1+x_2)*\psi_-(x+x_1)*\psi(x+x_1)*\psi_-(x).
\end{equation*}
Lemma \ref{welldefined} tells us that $d(\theta(x_1+x_2),\theta(x_1+x_2))\leq 75\eta$. Since the distance is a bilinear function (in the sense that it commutes with expectations), this implies that
\begin{equation*}
\E_{x,y}d(\psi(x+x_1+x_2)*\psi_-(x),\psi(y+x_1+x_2)*\psi_-(y))\leq 75\eta.
\end{equation*}
We are trying to estimate the distance
\begin{equation}\label{eqn2}
d(\E_x\psi(x+x_1+x_2)*\psi_-(x),\E_y\psi(y+x_1+x_2)*\psi_-(y+x_1)*\psi(y+x_1)*\psi_-(y)).
\end{equation}
By the bilinearity of $d$, it equals
\[\E_{x,y}d(\psi(x+x_1+x_2)*\psi_-(x),\psi(y+x_1+x_2)*\psi_-(y+x_1)*\psi(y+x_1)*\psi_-(y)),\]
which equals
\[\E_{x,y}d(\psi(x+x_1+x_2)*\psi_-(x)*\psi_-(y+x_1+x_2)*\psi(y),\psi_-(y+x_1)*\psi(y+x_1)).\]
By Lemma \ref{continuity}, this is at most the sum of
\[\E_{x,y}d(\psi(x+x_1+x_2)*\psi_-(x)*\psi_-(y+x_1+x_2)*\psi(y),\delta_0),\]
and
\[5\E_yd(\psi_-(y+x_1)*\psi(y+x_1),\delta_0)+10\beta,\]
where $\beta=(1+\eta)^4-1$. 
The first of these terms is equal to
\[\E_{x,y}d(\psi(x+x_1+x_2)*\psi_-(x),\psi(y+x_1+x_2)*\psi_-(y)),\]
which, as we have already remarked, is at most $75\eta$. We also have that
\[\E_yd(\psi_-(y+x_1)*\psi(y+x_1),\delta_0)=d(\E_y\psi_-(y+x_1)*\psi(y+x_1),\delta_0)=d(\theta(0),\delta_0),\]
which, by Corollary \ref{thetaofnought}, is at most $77\eta$. Therefore, the distance (\ref{eqn2}) is at most $75\eta+385\eta+10\beta<700\eta$.

We shall now further approximate 
\[\E_x\psi(x+x_1+x_2)*\psi_-(x+x_1)*\psi(x+x_1)*\psi_-(x)\]
by
\[\theta(x_2)*\theta(x_1)=\E_{u,v}\psi(u+x_1+x_2)*\psi_-(u+x_1)*\psi(v+x_1)*\psi_-(v).\]
By Corollary \ref{lipschitzconvolutions} with $\beta=2\eta+\eta^2$, and using the bilinearity of $d$, the distance between them is at most
\[(1+\beta)^2\E_{x,u}d(\psi(x+x_1+x_2)*\psi_-(x+x_1),\psi(u+x_1+x_2)*\psi_-(u+x_1))\]
\[+(1+\beta)^2\E_{x,v}d(\psi(x+x_1)*\psi_-(x),\psi(v+x_1)*\psi_-(v))+8\beta^2.\]
But this equals $(1+\beta)^2(d(\theta(x_2),\theta(x_2))+d(\theta(x_1),\theta(x_1)))+8\beta^2$,
which, by Lemma \ref{welldefined}, is at most $(1+\beta)^2.150\eta+8\beta^2$. Since $(1+\beta)^2=(1+\eta)^4\leq 2$, this is at most $400\eta$.
\end{proof}

Now let us put together what we have proved so far.

\begin{corollary} \label{thetatogamma}
There exists a group homomorphism $\gamma:G\to H$ such that $d(\theta(x),\delta_{\gamma(x)})\leq 80\eta$ for every $x\in G$. 
\end{corollary}

\begin{proof}
Corollary \ref{thetaconcentrated} is the statement that there exists a \textit{function} $\gamma:G\to H$ with the property stated. It remains to show that $\gamma$ is a group homomorphism. 

Let $\beta=(1+\eta)^2-1$.  Then, by Corollary \ref{lipschitzconvolutions},
\[d(\theta(x_1)*\theta(x_2),\delta_{\gamma(x_1)}*\delta_{\gamma(x_2)})\leq(1+\beta)^2\bigl(d(\theta(x_1),\delta_{\gamma(x_1)})+d(\theta(x_2),\delta_{\gamma(x_2)})\bigr)+8\beta^2,\]
which is at most $(1+\beta)^2.160\eta+8\beta^2\leq 400\eta$.

By Corollary \ref{thetaconcentrated}, Lemma \ref{homomorphism}, Lemma \ref{continuity} (twice) and this calculation,
\begin{align*}
d(\delta_{\gamma(x_1+x_2)},\delta_{\gamma(x_1)}*\delta_{\gamma(x_2)}) & \leq d(\theta(x_1+x_2),\delta_{\gamma(x_1)}*\delta_{\gamma(x_2)}) + 5 d(\theta(x_1+x_2), \delta_{\gamma(x_1+x_2)})+ 10\beta\\
& \leq d(\theta(x_1+x_2),\delta_{\gamma(x_1)}*\delta_{\gamma(x_2)}) + 400\eta + 10\beta\\
& \leq d(\theta(x_1+x_2),\theta(x_1)*\theta(x_2)) + 5 d(\theta(x_1)*\theta(x_2),\delta_{\gamma(x_1)}*\delta_{\gamma(x_2)}) + 400 \eta + 20 \beta\\
& \leq 1100 \eta + 2000 \eta + 400 \eta + 20 \beta \leq 4000 \eta.
\end{align*}
But $\delta_{\gamma(x_1)}*\delta_{\gamma(x_2)}=\delta_{\gamma(x_1)+\gamma(x_2)}$, so if $\eta < 1/4000$, the only way this estimate can be true is if $\gamma(x_1+x_2)=\gamma(x_1)+\gamma(x_2)$. 
\end{proof}

It remains to relate the homomorphism $\gamma$ to the original function $\phi:U\to H$. This we do with a standard averaging argument. 

\begin{theorem}
There exists an affine homomorphism $\alpha:G\to H$ such that $\phi(x)=\alpha(x)$ for all but $80\eta|U|$ elements $x\in U$.
\end{theorem}

\begin{proof}
First, let us write an expression for $\theta(x)$. It is given by
\[\theta(x)(h)=\E_{x_1+x_2+x_3-x_4-x_5-x_6=x}\mu(x_1)\dots\mu(x_6)\mathbf{1}_{\phi(x_1)+\phi(x_2)+\phi(x_3)-\phi(x_4)-\phi(x_5)-\phi(x_6)=h}.\]
Therefore, we have just shown that
\[\E_{x_1+x_2+x_3-x_4-x_5-x_6=x}\mu(x_1)\dots\mu(x_6)\mathbf{1}_{\phi(x_1)+\phi(x_2)+\phi(x_3)-\phi(x_4)-\phi(x_5)-\phi(x_6)=\gamma(x)}\]
is at least $1-80\eta$.

Taking the expectation over $x$, we deduce that
\[\E_{x_1,\dots,x_6}\mu(x_1)\dots\mu(x_6)\mathbf{1}_{\phi(x_1)+\phi(x_2)+\phi(x_3)-\phi(x_4)-\phi(x_5)-\phi(x_6)=\gamma(x_1+x_2+x_3-x_4-x_5-x_6)}\]
is also at least $1-80\eta$. Therefore (since $\mu$ is a probability measure on $G$), there exist $x_2,\dots,x_6$ such that, writing $z=x_2+x_3-x_4-x_5-x_6$ and $h=\phi(x_2)+\phi(x_3)-\phi(x_4)-\phi(x_5)-\phi(x_6)$, we have
\[\E_{x_1}\mu(x_1)\mathbf{1}_{\phi(x_1)=\gamma(x_1+z)-h}\geq 1-80\eta.\]
That is, $\phi(x)=\gamma(x)+\gamma(z)-h$ for all but $80\eta|U|$ elements of $U$. Thus, we may take $\alpha(x)=\gamma(x)+\gamma(z)-h$.
\end{proof}

\section{Freiman homomorphisms that are affine almost everywhere are affine everywhere}

The final step in the argument is to prove that if $U$ is a random set where every element is chosen independently with probability $C n^{-2/3} (\log n)^{1/3}$, then with high probability every Freiman homomorphism defined on $U$ that is affine on at least $99.99\%$ of $U$ is in fact affine on all of $U$. 

\begin{definition}
Let $U$ be a subset of an Abelian group $G$. An \emph{affine homomorphism} $\phi:U\to H$ from $U$ to an Abelian group $H$ is a function of the form $\phi(u)=a+\psi(u)$, where $\psi:U\to H$ is the restriction to $U$ of a group homomorphism. The set $U$ has the $(1-\eta)$-\emph{extension property} if every Freiman homomorphism from $U$ to an Abelian group $H$ that coincides with an affine homomorphism $\phi$ on a subset $V\subset U$ of size at least $(1-\eta)|U|$ is in fact equal to $\phi$.
\end{definition}

\begin{definition}
Let $U$ be a subset of a finite Abelian group $G$, let $W\subset U$ and let $V=U\setminus W$. Then $W$ is \emph{additively isolated in} $U$ if $(V+V-V)\cap W=\emptyset$. Let $\eta>0$. The set $U$ is $(1-\eta)$-\emph{additively connected} if no subset of $U$ of size at most $\eta|U|$ is additively isolated in $U$. 
\end{definition}

The definitions we have given are not standard, but they are the ones that we shall need for our argument. Let us prove a simple lemma that illustrates their usefulness.

\begin{lemma}
Let $G$ be a finite Abelian group and let $\eta>0$. Then every $(1-\eta)$-additively connected subset $U$ of $G$ has the $(1-\eta)$-extension property.
\end{lemma}

\begin{proof}
Let $\phi$ be a Freiman homomorphism defined on $U$, let $V_0$ be a subset of $U$ of size at least $(1-\eta)|U|$ and suppose that the restriction of $\phi$ to $V_0$ is also the restriction to $V_0$ of an affine homomorphism $\psi:G\to H$. Let $V$ be the set of all points $u\in U$ such that $\phi(u)=\psi(u)$.

If $V$ is not the whole of $U$, then we have an easy contradiction, since $|V|\geq(1-\eta)|U|$, from which it follows by hypothesis that its complement $W$ is not additively isolated in $U$. Therefore, we can find $v_1,v_2,v_3\in V$ such that $v_1+v_2-v_3=w\in W$. Since $\phi$ is a Freiman homomorphism, it follows that $\phi(w)=\phi(v_1)+\phi(v_2)-\phi(v_3)$, which equals $\psi(v_1)+\psi(v_2)-\psi(v_3)$, and since $\psi$ is an affine homomorphism, this is equal to $\psi(w)$. 
\end{proof}

The converse of this statement is not quite true, because it is also possible for $W+W-W$ to intersect $V$. We leave it as an exercise to find a counterexample.

Our strategy, then, is to prove that if $G$ is a finite Abelian group of order $n$, then a binomial random subset $U\subset G$ where each element is chosen independently with probability $C n^{-2/3} (\log n)^{1/3}$  is $(1-\eta)$-additively connected with high probability for some absolute constant $\eta>0$. In what follows, we will work in a slightly different probabilistic model, proving that random subsets $U$ of $G$ with fixed size $C(n\log n)^{1/3}$ are $(1-\eta)$-additively connected with high probability for some appropriate $\eta>0$. However, this easily implies that the same holds in the binomial model.

We begin with a preparatory lemma.

\begin{lemma} \label{azumalemma}
Let $G$ be a finite Abelian group of order $n$, let $K$ be a subset of $G$ of size $k$ and let $A$ be a random subset of $G$ of size $m$. Then, if $km\leq n$, $\P[|A+K|<km/4]\leq e^{-m/32}$, and if $n\leq km\leq 2n$, $\P[|A+K|<n/2]\leq e^{-m/800}$.
\end{lemma}

\begin{proof}
Let $x$ be any element of $G$. The probability that $x\in A+K$ is at least $1-(1-k/n)^m\geq 1-e^{-km/n}$. Therefore, the expectation of $|A+K|$ is at least $n(1-e^{-km/n})$. If we alter a single element of $A$, then the change to $|A+K|$ is at most $k$. It follows by Azuma's inequality that 
\[\P[|A+K| < \E|A+K| - tk]\leq e^{-t^2/2m}.\]
Now $e^{-x}\leq 1-x+x^2/2$ for every $x\geq 0$, so if $km\leq n$, then $n(1-e^{-km/n})\geq n(km/n-k^2m^2/2n^2)\geq nkm/2n=km/2$. Therefore, $\P[|A+K|<km/4]\leq e^{-(m/4)^2/2m}$, which establishes the first bound. If $km\geq n$, then $n(1-e^{-km/n})\geq n(1-1/e)>3n/5$, so $\P[|A+K|<n/2]\leq e^{-(n/10k)^2/2m}\leq e^{-(m/20)^2/2m}$, which establishes the second.
\end{proof}

Now let $X$ be the set $\{1,2,\dots,t\}$ for a $t$ to be chosen later (which will be of the form $C(n\log n)^{1/3}$). Let $k\leq\eta t$ for some $\eta>0$ also to be chosen later (which will be an absolute constant). For each $B'\subset X$ of size $k$, let $A_1'(B'), A_2'(B')$ and $A_3'(B')$ be three sets of equal size that partition $X\setminus B'$, with the exception of at most two elements, with these sets chosen arbitrarily. Now let $\phi$ be a random function from $X$ to $G$. It is easy to show that for sufficiently large $n$ the probability that the restrictions of $\phi$ to the sets $B', A_1'(B'), A_2'(B')$ and $A_3'(B')$ are all injections is at least 1/2. If we condition on this event, then the images $B, A_1, A_2, A_3$ of those four sets are independent random subsets of $G$ of the appropriate cardinalities. If we condition further on $\phi$ being an injection (so now we ask for the images to be disjoint), which again is true with probability at least 1/2, then their union $U$ is a random set of size $t$ and $B$ is a random subset of $U$ of size $k$.

For each $B'\subset X$, let $P(B')$ be the probability that $A_1+A_2-A_3$ is disjoint from $B$, given that the restrictions of $\phi$ to the sets $B', A_1'(B'), A_2'(B')$ and $A_3'(B')$ are all injections. If $\sum_{|B'|\leq\eta t}P(B')=p$, then the probability that there exists $B'\subset X$ of size at most $\eta t$ such that $A_1+A_2-A_3$ is disjoint from $B$ is at most $p$. If we now condition on $\phi$ being an injection, this probability goes up to at most $2p$. Therefore, with probability at least $1-2p$, no subset of $U$ of size at most $\eta |U|$ is additively isolated in $U$. In other words, with probability at least $1-2p$, $U$ is $(1-\eta)$-additively connected.

We shall therefore concentrate our attention on estimating $P(B')$.

\begin{lemma} \label{unionbound}
Let $G$ be a finite Abelian group of order $n$, let $B$ be a fixed subset of $G$ of size $k$, and let $A_1, A_2$ and $A_3$ be random subsets of $G$ of size $s=C(n\log n)^{1/3}$. Let $t=k+3s$ and suppose that $k\leq t/10^5$. Then there exists an absolute constant $C > 0$ such that the probability that $A_1+A_2-A_3$ and $B$ are disjoint is at most $(2n)^{-2}\binom tk^{-1}$.
\end{lemma}

\begin{proof}
Observe first that $A_1+A_2-A_3$ and $B$ are disjoint if and only if $A_1+A_2-B$ and $A_3$ are disjoint. Next, note that by Lemma \ref{azumalemma} and the fact that $ks\leq n$ (we are assuming throughout that $n$ is sufficiently large), we have that $|A_1-B|\geq ks/4$ with probability at least $1-e^{-s/32}$. 

Let $K=A_1-B$. The rest of the proof splits into two cases. If $|K|s\leq n$, then Lemma \ref{azumalemma} implies that $|A_2+K|\geq|K|s/4\geq ks^2/16$ with probability at least $1-e^{-s/32}$. If this event happens, then the probability that $A_1+A_2-B$ is disjoint from $A_3$ is at most $(1-ks^2/16n)^s\leq\exp(-ks^3/16n)=n^{-C^3k/16}$. Since $k\leq t/10^5$, standard estimates for binomial coefficients give us, with room to spare, that $\binom tk\leq (et/k)^k \leq e^{t/4000} \leq e^{s/1000}$. Also, when $C = 4$, $n^{-C^3k/16}$ is much less than $\binom nk^{-1}$. Together, these estimates easily suffice to show that 
\[2e^{-s/32} + n^{-C^3k/16} \leq (2n)^{-2} \binom{t}{k}^{-1}.\]

If $|K|s>n$, then let $L$ be a subset of $K$ such that $n\leq |L|s\leq 2n$. The second part of Lemma \ref{azumalemma} implies that $\P[|A_2+L|<n/2]\leq e^{-s/800}$. If $|A_2+L|\geq n/2$, then the probability that $A_2+L$ is disjoint from $A_3$ is at most $2^{-s}$, which is much smaller than $\binom tk^{-1}$ when $k=t/10^5$, by the estimates in the previous paragraph. Moreover, by the same estimates, $e^{-s/800}$ is much less than $\binom tk^{-1}$, so we are again done in this case provided $C$ is sufficiently large. 
\end{proof}

Combining Lemma \ref{unionbound} with the preceding remarks, we obtain the main result of this section.

\begin{lemma}
Let $G$ be a finite Abelian group of order $n$ and let $U$ be a random subset of $G$ of size $C(n\log n)^{1/3}$. Then there exists an absolute constant $C > 0$ such that the probability that $U$ is $(1-10^{-5})$-additively connected is at least $1-1/n$.
\end{lemma}

\noindent Since this was what we needed to complete the proof of our main theorem, we are now done.

\section{Concluding remarks}

Ultimately, one might hope to prove a more precise result still. Define the \textit{Freiman dimension} of a subset $A$ of an Abelian group to be one less than the dimension of the vector space of all Freiman homomorphisms from $A$ to $\R$. For example, $\Z_N$ has Freiman dimension 0, since every Freiman homomorphism from $\Z_N$ to $\R$ is constant, and an arithmetic progression $P\subset\Z$ has Freiman dimension 1, since a Freiman homomorphism from $P$ to $\R$ is determined by the values it takes at the first two points. If $A$ is a subset of $\Z_N$ such that every Freiman homomorphism from $A$ to an Abelian group $H$ extends to a Freiman homomorphism defined on all of $\Z_N$, then $A$ has Freiman dimension 0, since if $H=\R$, then the extension, and therefore the original homomorphism, must be constant. Therefore, Theorem~\ref{thm:main2} implies that if a random subset $A$ of $\Z_N$ is chosen with probability $C(\log N)^{1/3} N^{-2/3}$, then it has Freiman dimension $0$. It would be interesting to understand how the Freiman dimension of the random set $A$ decreases as the probability moves from $CN^{-2/3}$ to $C(\log N)^{1/3} N^{-2/3}$. For example, it might be the case that a hitting time result holds. That is, if we build up our random set one element at a time, it may be that the Freiman dimension drops to $0$ at precisely the same moment when all elements are contained in an additive quadruple.

\end{document}